\documentclass[11pt,reqno, a4paper]{amsart}

\usepackage{amsmath}
\usepackage{enumitem}   
\usepackage{amssymb}
\usepackage{amscd}
\usepackage{amsfonts}
\usepackage{mathrsfs}
\usepackage{dsfont}
\usepackage{setspace}
\usepackage{version}
\usepackage{mathtools}
\usepackage{soul,tikz}
\usepackage{tikz-cd}

\usepackage{enumitem}
\usepackage{pgfplots,wrapfig}

\usepackage{xcolor}

\definecolor{crimson}{rgb}{0.85, 0.08, 0.4}
\definecolor{bleudefrance}{rgb}{0.2, 0.5, 0.9}
\usepackage[pdftex,colorlinks,linkcolor=crimson,citecolor=bleudefrance,filecolor=black]{hyperref}
\usepackage[utf8]{inputenc}

\newtheorem{theorem}{Theorem}[section]
\newtheorem{lemma}[theorem]{Lemma}
\newtheorem{corollary}[theorem]{Corollary}
\newtheorem{proposition}[theorem]{Proposition}

\theoremstyle{definition}
\newtheorem{definition}[theorem]{Definition}
\newtheorem{example}[theorem]{Example}

\theoremstyle{remark}
\newtheorem{remark}[theorem]{Remark}

\numberwithin{equation}{section}

\renewcommand{\leq}{\leqslant}
\renewcommand{\geq}{\geqslant}
\newsavebox{\proofbox}
\savebox{\proofbox}{\begin{picture}(7,7)  \put(0,0){\framebox(7,7){}}\end{picture}}

\newcommand\E{\mathbb{E}}

\newcommand\R{\mathbb{R}}

\newcommand\C{\mathbb{C}}
\newcommand\N{\mathbb{N}}

\newcommand\im{\operatorname{Im}}

\newcommand\supp{\operatorname{supp}}
\newcommand\GL{\operatorname{GL}}
\newcommand\PP{\operatorname{P}}

\newcommand\PGL{\operatorname{PGL}}
\newcommand\SO{\operatorname{SO}}

\newcommand\Endo{\operatorname{End}}

\newcommand\Q{\mathbb{Q}}

\newcommand{\efface}[1]{}

\usepackage{pgfplots}
\pgfplotsset{compat=1.8}

\begin{document}

\author{Richard Aoun}
\address{University Gustave Eiffel, Champs-sur-Marne, 
5 boulevard Descartes Champs-sur-Marne 77420 Marne-la-Vallée Cedex 2, France}
\email{richard.aoun@univ-eiffel.fr}

\author{Cagri Sert}
\address{Institut f\"{u}r Mathematik, Universit\"{a}t Z\"{u}rich, 190, Winterthurerstrasse, 8057 Z\"{u}rich, Switzerland}
\email{cagri.sert@math.uzh.ch}
\thanks{C.S. is supported by SNF Ambizione grant 193481}

\subjclass[2010]{Primary 37H15; Secondary 60J05, 60B15, 37A20}
\keywords{Random walks, stationary measures, projective space,  random matrix products}

\title[Stationary measures on projective spaces]{Stationary probability measures on projective spaces 2:  the critical case}

\begin{abstract}
In a previous article, given a finite-dimensional real  vector space $V$ and a probability measure $\mu$ on $\PGL(V)$ with finite first moment, we gave a description of all $\mu$-stationary probability measures on the projective space $\PP(V)$ in the non-critical (or Lyapunov dominated) case. In the current article, we complete the analysis by providing a full description of the more subtle critical case. Our results demonstrate an algebraic rigidity in this situation. Combining our results with those of Furstenberg--Kifer ('83), Guivarch--Raugi ('07) $\&$ Benoist--Quint ('14), we deduce a classification of all stationary probability measures on the projective space for i.i.d random matrix products with finite first moment without any algebraic assumption.
\end{abstract}

\setcounter{tocdepth}{1}
\maketitle
\section{Introduction}

Let $V$ be a finite-dimensional real
vector space and $\mu$ a probability measure on $\GL(V)$. Let $(X_n)_{n\in \N}$ denote a $\GL(V)$-valued sequence of iid random variables with distribution $\mu$ and write $L_n=X_n \ldots X_1$ for the associated random matrix product. Via the action of $\GL(V)$ on the projective space $\PP(V)$ the random product $L_n$ induces a Markov chain on $\PP(V)$.
The goal of the paper is to give a classification of all stationary probability measures of this Markov chain. In the sequel, we will refer to these as \textit{$\mu$-stationary}, or simply, stationary measures on $\PP(V)$.

Let us start with a brief history. The study of stationary measures on projective spaces was initiated by Furstenberg--Kesten \cite{furstenberg.kesten} and Furstenberg \cite{furstenberg.noncommuting} who realized early on that these encode to a
great extent the asymptotic behaviour of random matrix products. Indeed, using stationary measures, Furstenberg proved the key result of positivity of top Lyapunov exponent and he gave a formula for the Lyapunov exponents (Furstenberg formula). He also showed the uniqueness of stationary measure in the irreducible and proximal case. Under the algebraic semisimplicity assumption, extending the work of Furstenberg, the full description of stationary measures was obtained by Guivarc'h--Raugi \cite{guivarch.raugi.ens} and Benoist--Quint \cite{BQ.projective}. Their result demonstrates an algebraic rigidity phenomenon in the semisimple situation and it shows that in the semisimple case every $\mu$-stationary and ergodic probability measure on $\PP(V)$ is $\mu$-homogeneous, a notion which will make precise below.

The study of stationary measures without the algebraic semisimplicity assumption -- a setting which encompasses many familiar examples such as Bernoulli convolutions or more generally self-affine measures -- was tackled by Furstenberg--Kifer \cite{furstenberg.kifer}  and Hennion \cite{hennion}  who gave a first description of stationary measures without additional algebraic assumptions. Crucially, they used this classification to prove continuity of Lyapunov exponents 
in the critical case, the setting which  corresponds precisely to the setting of our Theorem \ref{thm.main.Lmu=0} below (see also Peres \cite{peres} for further applications). Broadly speaking, without semisimplicity assumptions, for a general classification of stationary measures in this context, there are two dynamically distinguished cases that one needs to deal with. These are the so-called non-critical (or Lyapunov-dominated) and critical cases. Non-criticality refers to a relative contracting or expanding dynamics whereas the more subtle critical case is characterized by absence of difference in dynamical behaviour, or expansion rates, in different invariant subspaces.

Due to the generality of projective dynamics without algebraic assumptions, stationary measures on projective spaces were studied in various special cases of interest: For the critical case, after the pioneering work of Furstenberg \cite{furstenberg.noncommuting}, Kesten--Spitzer \cite{kesten-spitzer} treated the case of non-negative matrices and Bougerol \cite{bougerol.tightness} gave a full description of stationary measures on vector spaces. The case of affine random walks (also called affine stochastic recursion) was brought to conclusion by Bougerol--Picard \cite{bougerol.picard}. These works constitute the first important situations where a complete description in the critical case were achieved. More recently, Benoist--Bru\`{e}re \cite{benoist.bruere} classified stationary probability measures on affine Grassmannians under a Zariski-density assumption.

The goal of the current paper is to give a classification of stationary measures in the critical case, and combine this with our previous work \cite{aoun.sert.stationary1} on the non-critical case to obtain a full description without any assumptions on the random matrix products (except finite first moment assumption). We will prove a rigidity result, showing that in the critical case all ergodic stationary measures are supported on semisimple subspaces (Theorem \ref{thm.main.Lmu=0}). Combined with the aforementioned work of Guivarc'h--Raugi $\&$ Benoist--Quint, this implies that in the critical case all $\mu$-stationary ergodic measures are $\mu$-homogeneous (Corollary \ref{corol.critical.homogene}). Finally, combining this with the results of Furstenberg--Kifer and our previous work, we obtain a description of $\mu$-stationary measures (Theorem \ref{thm.main}): given any probability measure $\mu$ with finite first moment, any $\mu$-stationary ergodic probability measure on $\PP(V)$ is the unique lift of a $\mu$-homogeneous stationary measure.

We now proceed to make all these statements precise.

\bigskip

A probability measure $\mu$ on $\GL(V)$ is said to have \textit{finite first  moment} if 
$$\int\log\max \{\|g\|, \|g^{-1}\|\}
\,d\mu(g)<\infty,$$ where $\|\cdot\|$ is any norm on $\Endo(V)$. Given a probability measure $\mu$ with finite first moment, the top Lyapunov exponent of $\mu$ is the constant, denoted $\lambda_1(\mu)$, which is the almost-sure limit of $\lim_{n\to +\infty}\frac{1}{n}\log \|X_n\cdots X_1\|$. The existence of the  limit is guaranteed by Kingman's subadditive ergodic theorem and was first proved in this setting by Furstenberg--Kesten \cite{furstenberg.kesten}. We denote by $\Gamma_{\mu}$ the closed semigroup of $\GL(V)$ generated by the support of $\mu$. If $W$ is a $\Gamma_\mu$-invariant subspace of $V$, we will also denote by $\lambda_1(W)$ the top Lyapunov exponent of the $\mu$-random matrix products induced on $W$. In particular, with this notation, $\lambda_1(\mu)=\lambda_1(V)$. Our first main result reads as follows:

\begin{theorem} \label{thm.main.Lmu=0}(Critical case)
Let $\mu$ be a probability measure on $\GL(V)$ with   finite first   moment. Assume that the top Lyapunov of any non-zero $\Gamma_\mu$-invariant subspace  of $V$ is equal to $\lambda_1(\mu)$. Then the action of $\Gamma_{\mu}$ on the subspace generated by the support of any $\mu$-stationary probability measure on $\PP(V)$ is semisimple. 
\end{theorem}

When the $\Gamma_\mu$-action on a vector space $W$ is semisimple (equivalently when the Zariski-closure of $\Gamma_\mu$ in $\GL(W)$ is reductive), Guivarc'h--Raugi \cite{guivarch.raugi.ens} and Benoist--Quint \cite{BQ.projective} proved that any $\mu$-stationary ergodic probability measure $\nu$ is $\mu$-\textit{homogeneous} in the sense that the Zariski-closure of the support of $\nu$ is a closed orbit of $H_\mu$ in $\PP(W)$, where $H_\mu$ denotes the Zariski closure of $\Gamma_\mu$ in $\GL(W)$. Moreover they showed that this orbit supports a unique $\mu$-stationary probability measure. This implies that the irreducible components of the subspace $V_\nu$ generated by the support of $\nu$ have the same weight (in particular $V_\nu$ is critical) so that the hypothesis of Theorem \ref{thm.main.Lmu=0} is in some sense sharp. On the other hand, their results combined with Theorem \ref{thm.main.Lmu=0} imply the following. 

\begin{corollary}\label{corol.critical.homogene}
Keep the assumptions of Theorem \ref{thm.main.Lmu=0}. Then any $\mu$-stationary ergodic probability measure is $\mu$-homogeneous,  and there is a natural bijection 
$$
\{\mu\text{-stationary and ergodic measures on $\PP(V)$}\} \simeq \{ \text{Compact} \; H_\mu\text{-orbits in} \; \PP(V)\}. 
$$
\end{corollary}

Corollary \ref{corol.critical.homogene} combined with the works of Furstenberg--Kifer \cite{furstenberg.kifer} and Hennion \cite{hennion} and our previous work \cite{aoun.sert.stationary1} yields the following  classification of stationary probability measures on $\PP(V)$. 

\begin{theorem}[General Classification]\label{thm.main}
Let $\mu$ be a probability measure on $\GL(V)$ with  finite first   moment. Then there exists a filtration $V=F_1 \supsetneq F_2 \cdots \supsetneq F_k \supsetneq \{0\}$ of $\Gamma_{\mu}$-invariant subspaces such that for each $\mu$-stationary ergodic probability $\nu$, there exists $i\in \{1,\cdots, k\}$ such that $\nu(\PP(F_i)\setminus \PP(F_{i+1}))=1$ and $\nu$ is the unique lift in 
$\PP(F_i)\setminus \PP(F_{i+1})$ of a $\mu$-homogeneous stationary measure in   $\PP(F_i/F_{i+1})$. 
\end{theorem}

Here and elsewhere, when $W<V$ is an invariant subspace, we employ the term \textit{lift} to mean that $\nu$ is a $\mu$-stationary probability measure on $\PP(V) \setminus \PP(W)$ whose pushforward on $\PP(V/W)$ under the map induced by the natural projection $V \to V/W$ is $\overline{\nu}$.

In the above statement, the fact that each ergodic stationary probability measure $\nu$ lives in some $\PP(F_i)\setminus \PP(F_{i+1})$ is contained in Furstenberg--Kifer and Hennion's work. The new information given by Theorem \ref{thm.main} is in the description of these measures: the projection of each such measure $\nu$ on $\PP(F_i/F_{i+1})$ is one of the stationary measures described by Guivarc'h--Raugi and Benoist--Quint (i.e.~ $\mu$-homogeneous) and $\nu$ is the unique lift thereof on the open subset $\PP(F_i)\setminus \PP(F_{i+1})$.
In particular,  combined with the unique ergodicity results of \cite{guivarch.raugi.ens,BQ.projective} for $\mu$-homogeneous measures, this yields a parametrization of $\mu$-stationary and ergodic probability measures by compact $H_\mu$-orbits on the quotients of Furstenberg--Kifer--Hennion spaces:
\begin{corollary}
$$\{\mu\text{-stationary and ergodic probabilities on $\PP(V)$}\} \simeq \bigcup_{i=1}^k \, \{\text{Compact $H_\mu$-orbits in} \, \PP(F_i/F_{i+1})\}. $$
\end{corollary}

\begin{proof}[Proof of Theorem \ref{thm.main}]
Let $\nu$ be a $\mu$-stationary and ergodic probability measure. By Furstenberg--Kifer \cite{furstenberg.kifer}, there exists a filtration $V=F_1 \supsetneq F_2 \cdots \supsetneq F_k \supsetneq \{0\}$ of $\Gamma_{\mu}$-invariant subspaces and $i \in \{1,\ldots,k\}$ such that $\nu(\PP(F_i))=1$ and $\nu(\PP(F_{i+1} ))=0$. Therefore, $\nu$ projects down to a stationary ergodic measure $\overline{\nu}$ on $\PP(F_i/F_{i+1})$. By \cite[Lemma 3.7]{furstenberg.kifer}, the quotient $F_i/F_{i+1}$ satisfies the assumption of Theorem \ref{thm.main.Lmu=0} (i.e.~ $F_2(F_i/F_{i+1})=\{0\}$) and therefore Theorem \ref{thm.main.Lmu=0} applies and it yields that $\overline{\nu}$ is supported on a $\Gamma_\mu$-completely reducible subspace. Now Guivarc'h--Raugi \cite{guivarch.raugi.ens} $\&$ Benoist--Quint \cite{BQ.projective} applies and shows that $\overline{\nu}$ is a $\mu$-homogeneous stationary measure. 
Since $\lambda_1(F_{i+1})<\lambda_1(F_i/F_{i+1})$, \cite[Theorem 1.1]{aoun.sert.stationary1} implies that $\nu$ is the unique lift of the $\mu$-homogeneous stationary measure $\overline{\nu}$, concluding the proof.
\end{proof}

We note in passing that unlike many results for limit theorems in random matrix products theory, Theorem \ref{thm.main.Lmu=0} is specific to vector spaces over archimedean local fields (so $\R$ or $\C$\footnote{Indeed, Theorem \ref{thm.main.Lmu=0} holds for a complex vector space; this follows easily from the real case.}); it fails for non-archimedean local fields (finite extensions of $\Q_p$ or $\mathbb{F}_p((T))$). Indeed, this result is ultimately a generalization of the fact that a random walk on $\R$ has a stationary probability measure if and only if it is trivial (i.e.~ the law of the random walk is the Dirac measure on $0$) and already this statement obviously fails for non-archimedean local fields.

\bigskip

In specific models, such as the ones coming from homogeneous dynamics as random walks on affine spaces \cite{bougerol.picard} or affine Grassmanians \cite{benoist.bruere}, it is desirable to have a formulation in terms of  lifts with respect to a specific invariant subspace. More precisely, given a $\Gamma_{\mu}$-invariant subspace $W$ of $V$, one is interested in the existence of $\mu$-stationary probability measures on $\PP(V)\setminus \PP(W)$. For instance, a direct consequence of Theorem \ref{thm.main.Lmu=0} is that if $\lambda_1(V)=\lambda_1(W)$ and $W$ and $V/W$ are irreducible, then there is no invariant stationary probability measure on $\PP(V) \setminus \PP(W)$ unless $W$ has an invariant complement. Treating the general case (i.e.~ without irreducibility assumptions) requires some more care and is done in the following result by combining Theorem \ref{thm.main.Lmu=0}, together with \cite{furstenberg.kifer} and the contracting case \cite{aoun.sert.stationary1}.

\begin{theorem}[Lift with respect to a given subspace]\label{thm.with.W}
Let $\mu$ be a probability measure on $\GL(V)$ with finite first moment and $W$ a $\Gamma_\mu$-invariant subspace. Let $\overline{\nu}$ be a $\mu$-stationary and ergodic probability measure on $\PP(V/W)$ such that $\lambda_1(V_{\overline{\nu}})\leq \lambda_1(W)$. 
Then, the following are equivalent: 
\begin{itemize}
    \item[(i)] There exists a $\mu$-stationary lift $\nu$ of $\overline{\nu}$ on $\PP(V) \setminus \PP(W)$. 
\item[(ii)]
There exists a $\Gamma_\mu$-invariant subspace $W'<V$ such that $\lambda_1(W' \cap W)<\lambda_1(V_{\overline{\nu}})$ and $\PP(V_{\overline{\nu}})=\PP(W'/ W'\cap W)$. 
\end{itemize}
In this case, we have $\lambda_1(V_{\overline{\nu}})=\lambda_1(V_\nu)$.
\end{theorem}

This theorem generalizes Bougerol's result \cite[Theorem 5.1]{bougerol.tightness} and in particular  Bougerol--Picard \cite{bougerol.picard}   in the invertible case. It also recovers    Benoist--Bru\`{e}re's  results \cite[Theorem 1.6]{benoist.bruere}. 
We note that, unlike the expanding case $\lambda_1(V_{\overline{\nu}})<\lambda_1(W)$ (see \cite[Theorem 1.5]{aoun.sert.stationary1}), the existence of $\mu$-stationary lift does not imply its uniqueness. We refer to Remark \ref{rk.not.unique} for a more detailed explanation.

\bigskip

Finally, it is well-known, thanks to a classical result of Chevalley, that for any real algebraic group $G$ and algebraic subgroup $R$, the algebraic homogeneous space $G/R$ can be realized as  a $G$-orbit in $\PP(V)$ via some representation $G \to \GL(V)$. On the other hand, since $H_\mu$-orbits are locally closed in $\PP(V)$, any $\mu$-stationary ergodic probability measure is supported by a single $H_\mu$-orbit $\mathcal{O}$. Therefore, given a probability measure $\mu$ on $G<\GL(V)$, it is important to be able to describe $H_\mu$-orbits $\mathcal{O}\subseteq \PP(V)$ that support $\mu$-stationary measures. The following corollary that we obtain  by combining Theorem \ref{thm.main.Lmu=0}, \cite[Theorem 1.1]{aoun.sert.stationary1}, and \cite{BQ.projective} gives such a description.

\begin{corollary}\label{corol.homogeneous}
Let $\mu$ be a probability measure on $\GL(V)$ with a finite first moment.
Let $\mathcal{O}$ be a $H_\mu$-orbit in $\PP(V)$. Then the following hold.
\begin{enumerate}
\item 
The orbit $\mathcal{O}$ supports at most one $\mu$-stationary probability measure. 
\item Denote by $V_\mathcal{O}$ the linear space generated by the orbit $\mathcal{O}$ and by $F_\mu:=F_2(V_\mathcal{O})$ the maximal $H_\mu$-invariant subspace of slower expansion. Then the following are equivalent. 
\begin{enumerate}
    \item The orbit $\mathcal{O}$ supports a $\mu$-stationary probability measure.
    \item The image of $\mathcal{O}$ under the natural projection $\PP(V_\mathcal{O}) \setminus \PP(F_\mu) \to \PP(V_\mathcal{O}/F_\mu)$ is compact and $\overline{O} \setminus O \subseteq \PP(F_\mu)$.
\end{enumerate}
\end{enumerate}
\end{corollary}
We note that since $\PP(F_\mu) \cap \mathcal{O}=\emptyset$, the projection map above is well-defined. Finally, the subspace $F_\mu<V_\mathcal{O}$ can be trivial (equivalently $V_\mathcal{O}$ is critical) in which case (2) above boils down to Corollary \ref{corol.critical.homogene}.

\subsection{Outline of the proof}

We give an outline of the proof of Theorem \ref{thm.main.Lmu=0}. After certain simplifying preliminary reductions (such as restricting, via return times, to a Zariski-connected group and a non-degenerate stationary measure) carried out in \S \ref{section.preliminary}, the proof hinges on proving Theorem \ref{thm.key} which says that the existence of non-degenerate stationary measure implies that every $\Gamma_\mu$-invariant and irreducible subspace $W<V$ with $\lambda_1(W)=\lambda_1(V)$ has a $\Gamma_\mu$-invariant complement. To prove this, we start by decomposing a given stationary measure $\nu$ into random pieces $\nu_b$ living on a (projective) subspace $W_b$ of the most contracting (unstable) Oseledets space for the inverse random walk. Using the criticality assumption, the invariant space $W$ is shown to intersect the space $W_b$, and $W \cap W_b$ has co-dimension at least one in $W_b$. We then go on to construct a (random) idempotent quasi-projective transformation $\pi$ whose image is $W_b$. This key construction uses recurrence of Birkhoff sums (Lemma \ref{lemma.recurrence}) and some observations on random matrix products theory (Lemma \ref{lemma.Vr} $\&$ Lemma \ref{lemma.nondegenerate}). Using this transformation $\pi$, we construct a new (random) semigroup $\mathcal{S}_b$ acting on the space $W_b$ which, on the one hand, reads information from (a Zariski-dense subset of the) original semigroup $\Gamma_\mu$, and on the other hand, preserves the random probability measure $\nu_b$. By a classical result of Furstenberg, the latter property implies that the semigroup $\mathcal{S}_b$ acts on $W_b$ relatively compactly (all this is expressed in Proposition \ref{prop.ellis} which provides a crucial handle to construct a complement to $W$). Now, by relative compactness, we get an $\mathcal{S}_b$-invariant complement to $W\cap W_b$ inside $W_b$. By the aforementioned relation between $\mathcal{S}_b$ and $\Gamma_\mu$, we can then transfer this data to $\Gamma_\mu$. We finally finish the proof of Theorem \ref{thm.key} in an inductive way.

\subsection*{Acknowledgements}
The authors are thankful to Alex Eskin  for helpful remarks.

\section{Preliminary reduction}
\label{section.preliminary}
As mentioned in the introduction, the proof of Theorem \ref{thm.main.Lmu=0} will involve arguments making use of the Zariski topology on $\Gamma_\mu$. It will be more convenient at several places to work with a semigroup $\Gamma_\mu$ having a connected Zariski-closure, and also with stationary measures which are non-degenerate.   Accordingly, our goal in this section is to show that the following Theorem \ref{thm.inter} implies Theorem \ref{thm.main.Lmu=0}.

\begin{theorem}\label{thm.inter}
The conclusion of Theorem \ref{thm.main.Lmu=0} holds if we additionally assume that $\Gamma_\mu$ is Zariski-connected and $\nu$ is ergodic and non-degenerate in $\PP(V)$.  
\end{theorem}

Let us first fix once for all our notation and recall standard definitions and notions used above and in all the paper. All probability measures will be understood as Borel probability measures and all vector spaces will be real or complex. Given a probability measure $\mu$ on $\GL(V)$, denote by $\Gamma_\mu$ the closed semigroup of $\GL(V)$ generated by the support of $\mu$. Recall that a probability measure $\nu$ on the projective space $\PP(V)$ of $V$ is said to be $\mu$-\emph{stationary} if for every continuous function $f$ on $\PP(V)$, $\int{f dv}=\iint{f(g [x]) d\mu(g) d\nu([x])}$ where $[x]$ denotes the projection of a non-zero vector $x$ of $V$ on the projective space $\PP(V)$. It is said to be $\mu$-ergodic if it is extremal in the convex set of $\mu$-stationary probability measures on $\PP(V)$. We say that $\nu$ is \emph{non-degenerate in $PP(V)$} if $\nu(\PP(W))=0$ for any proper projective subspace $\PP(W)$ of $\PP(V)$. We denote by $V_\nu$ the linear subspace of $V$ generated by the support of $\nu$.

We recall now  standard facts about linear algebraic groups. A linear algebraic group $G$ is a subgroup of $\GL(V)$  for some finite-dimensional vector space $V$ which is closed for the Zariski topology, i.e.~ it is a subgroup of $\GL(V)$ such that the matrix entries of all elements of $\GL(V)$ satisfy a set of polynomial equations. A standard fact is that the Zariski-closure of a semigroup $\Gamma<\GL(V)$ is a linear algebraic group. We will denote by $G^o$ the connected component of $G$ for the Zariski-topology; it is an algebraic normal subgroup of $G$ of finite index. Recall finally that for linear algebraic groups  Zariski-connectedness is equivalent to (Zariski-)topological irreducibility (i.e.~ the intersection of any two open non-empty subsets is non-empty). Given a probability measure $\mu$ on $\GL(V)$, we denote by $H_\mu$ the Zariski-closure of  $\Gamma_\mu$.

All the random variables we consider will be defined on the probability space $(B,\mathcal{A},\beta)$ where  $B=\GL(V)^\N$, $\mathcal{A}$ the product Borel $\sigma$-algebra, and $\beta=\mu^{\otimes \N}$.
Given   a probability measure $\mu$ on $\GL(V)$ and $k \in \N$, we denote by $\mu^k$ the $k$-fold convolution $\mu \ast \cdots \ast \mu$ which is the distribution of the random variable $b=(b_i)_{i\in \N}\mapsto b_k \cdots b_1$. For convenience we put $\mu^0=\delta_{\textrm{id}}$. We denote by $\tau: B\to \N$, the stopping time defined by $b\mapsto \tau(b):=\inf\{k\in \N : b_k\cdots b_1\in H_\mu^o\}$  and by $\mu^\tau$ the distribution of the random variable $b\mapsto b_{\tau}\cdots b_1$. Note that $\tau$ is the hitting time of a state for the Markov chain given by the $\mu$-random walk on the finite group $H_\mu/H_\mu^o$. This Markov chain is irreducible since, by Zariski-density, $\Gamma_\mu$ surjects onto $H_\mu/H_\mu^o$, therefore $\tau$ is almost-surely finite.

In the proof of the implication Theorem \ref{thm.inter} $\implies$ Theorem \ref{thm.main.Lmu=0}, we will make use of the decomposition of $\nu$ given by the following result.

\begin{lemma}\label{lemma.finite}
Let $\nu$ be a $\mu$-stationary ergodic probability measure on $\PP(V)$. Then there exist finitely many probability measures $\nu_1,\cdots, \nu_k$ on $\PP(V)$, where $k:=[H_\mu : H_\mu^o]$,  such that 
\begin{itemize}
\item  $\nu=\frac{1}{k}\sum_{i=1}^k{\nu_i}$ 
\item  Each $\nu_i$ is a $\mu^{\tau}$-stationary ergodic probability  measure which is moreover non-degenerate in the projective subspace $V_{\nu_i}$ generated by its support.  
\end{itemize}

\end{lemma}
\begin{proof}
Let $g_1,\cdots, g_k\in H_\mu$ be such that  $H_\mu/H_\mu^o=\{g_i H_\mu^o : i=1,\cdots, k\}$. The group $H_\mu$ acts diagonally on $X:=H_\mu/H_\mu^o \times \PP(V)$; so that the $\mu$-random walk on $H_\mu$ induces a Markov chain on $X$. Namely the $n$th step of this Markov chain starting at $(g H_\mu^o, x)$ is $(b_n\cdots b_1 g H_\mu^{\circ}, b_n\cdots b_1 x)$.
For every $i\in \{1,\cdots, k\}$, denote $Y_i:=\{g_i H_\mu^{\circ}\}\times \PP(V)$. Since $H_\mu^o$ is normal in $H_\mu$, $H_\mu^o$ stabilizes each $Y_i$. 
Note that each  $Y_i$ is a  recurrent  subset of $X$  thanks to the irreducibility of the Markov chain on $H_\mu/H_\mu^o$. Moreover, for every $i\in \{1,\cdots, k\}$, the Markov chain on $Y_i$ induced by the Markov chain on $X$ projects on $\PP(V)$ to the $\mu^{\tau}$-Markov chain on $\PP(V)$.

By Chacon--Ornstein theorem, for $\nu$-a.e.~ every $x\in \PP(V)$, the sequence of probability measures $\frac{1}{n}\sum_{i=1}^n{\mu^i \ast \delta_x}$ converges weakly to $\nu$. Let $x$ be such a $
\nu$-generic point of $\PP(V)$.  
By compactness of $X$, the sequence of probability measures $\frac{1}{n}\sum_{i=1}^n{\mu^i \ast \delta_{(H_\mu^o, x)}}$ has a limit point $\eta$ which is a stationary probability measure for the Markov chain on $X$. Its projection on $H_\mu/H_\mu^o$ is $\mu$-stationary; it is then the uniform probability measure on $H_\mu/H_\mu^o$. Disintegrating $\eta$ over $H_\mu/H_\mu^o$, we get that 
$\eta=\frac{1}{k}\sum_{i=1}^k{ \delta_{g_i H_\mu^o}\otimes  \nu_i}$ where each $\nu_i$ is a probability measure on $\PP(V)$. In particular, $\nu=\frac{1}{k}\sum_{i=1}^k{\nu_i}$ and the restriction of $\nu$ to $Y_i$ is $\nu_i$.
Now it follows from \cite[Lemma 3.4]{BQ-III} that $\nu_i$ is stationary for the Markov chain restricted to $Y_i$, and that it is ergodic because $\nu$ is ergodic for the Markov chain on $X$. We deduce that each $\nu_i$ is $\mu^\tau$-stationary and ergodic. The fact that each $\nu_i$ is non-degenerate in $\PP(V_{\nu_i})$ is a general fact that we recall in Lemma \ref{lemma.connected.proper} below. 
\end{proof}

\begin{lemma}\label{lemma.connected.proper}
Let $\mu$ be a probability measure on $\GL(V)$ such that $\Gamma_\mu$ is Zariski-connected. Then each $\mu$-stationary  ergodic probability measure $\nu$ on $\PP(V)$ is non-degenerate in $\PP(V_{\nu})$. 
\end{lemma}

\begin{proof}
By standard arguments due to Furstenberg (see for instance the proof of \cite[Proposition 2.3]{bougerol.lacroix}), one can find a subspace $W$ of $V$ charged by $\nu$, of minimal dimension for this property, such that 
the orbit $\Lambda$ of $W$ by $\Gamma_{\mu}$  is finite. The set of elements $g\in \GL(V)$ that stabilize $\Lambda$ being Zariski-closed, we deduce that $H_\mu$ stabilizes also $\Lambda$.  Since $H_\mu$ is Zariski-connected, it must stabilize each  element of $\Lambda$ (otherwise it would have a proper finite index algebraic subgroup) so that $\Lambda=\{W\}$. By ergodicity of $\nu$, $V_{\nu}\subset W$. We conclude   by minimality of $W$. 
\end{proof}

\begin{proof}[Proof of Theorem \ref{thm.main.Lmu=0} (using Theorem \ref{thm.inter})]

\textbullet ${}$  (Ergodicity) Suppose each ergodic component $\nu_e$ of $\nu$ has the property that $V_{\nu_e}$ is completely reducible. Then, $V_\nu$ is completely reducible. Indeed, by finite-dimensionality,  $V_\nu$ is a sum of finitely many $V_{\nu_e}$'s. But a sum of completely reducible spaces is completely reducible.\\  
\indent \textbullet ${}$ (Non-degenerate $\nu$ and Zariski-connectedness) Given an ergodic $\mu$-stationary probability $\nu$, applying Lemma \ref{lemma.finite}, we obtain finitely many $\mu^\tau$-stationary and ergodic measures $\nu_1,\ldots,\nu_k$, each non-degenerate in projective subspace generated by its support and where $k=[H_\mu:H_\mu^o]$ and $\nu=\frac{1}{k}\sum_{i=1}^k\nu_i$. Therefore, $V_\nu=V_{\nu_1}+ \ldots + V_{\nu_k}$. We now claim that
\begin{enumerate}
\item $\overline{\Gamma}^Z_{\mu^\tau}=H_\mu^o$.
\item $\mu^{\tau}$ has finite first moment.
\item The measure $\mu^\tau$ has the property that every $\Gamma_{\mu^{\tau}}$-invariant subspace has the top Lyapunov exponent $\lambda_1(\mu^\tau)$.
\end{enumerate}
Once these claims are established, it follows from Theorem \ref{thm.inter} that each $V_{\nu_i}$ is $H_\mu^o$-completely reducible and hence $V_\nu$ is $H_\mu^o$-completely reducible. Since we are working in characteristic zero, the complete reducibility of $V_{\nu}$ as $H_\mu$-space is equivalent to its complete reducibility as $H_\mu^o$-space 
(see e.g.~ \cite[Lemma 3.1]{mostow.fully.reducible}).
    
It remains therefore to prove Claims (1), (2), and (3) above. For (1), it is enough to show that $\Gamma_{\mu^{\tau}}=\Gamma_{\mu}\cap H_\mu^o$. The inclusion $\subset$ is trivial. For the other inclusion, observe that $\Gamma_{\mu}\cap H_\mu^o$ is open in $\Gamma_\mu$ because $H_\mu^o$ has finite index in $H_\mu$. Hence for every   $g\in  \Gamma_{\mu}\cap H_\mu^o$ and every neighborhood $O$ of $g$, we have $\beta(\limsup \{b_n\cdots b_1\in O\})>0$ and hence $\beta(\limsup\{b_{\tau(n)}\cdots b_1\in O\})>0$. Hence   $g\in \Gamma_{\mu^{\tau}}$.\\[3pt]
For (2), this is showed in \cite[Corollary 5.6]{BQ.book}. 
\\[3pt]
Finally, we show (3). 
By \cite{furstenberg.kifer,hennion}, the assumption on $\mu$ is equivalent to saying that for every non-zero vector $x$ of $V$ we have  
$\lim_{n\to \infty}{\frac{1}{n}\log \|b_n\cdots b_1 x\|}= \lambda_1(\mu)$. By the law of large numbers, this implies that for every non-zero vector $x$ of $V$, almost surely $\lim_{n\to \infty}{\frac{1}{n}\log \|b_{\tau(n)}\cdots b_1 x}\|= \E(\tau)\lambda_1(\mu)$, where $\tau(n)$ is defined inductively by $\tau(n):=\tau(n-1) +\tau\circ \theta^{\tau(n-1)}$ and $\tau(0)=0$. Again, by  \cite{furstenberg.kifer,hennion}, this implies the desired property for $\mu^{\tau}$. 
\end{proof}

\section{Non-degenerate stationary probability measures on the projective space}

The goal of this section is to deduce some consequences of the existence of a non-degenerate $\mu$-stationary probability measure on the $\mu$-random walk and on the structure of the semigroup $\Gamma_\mu$ generated by the support of $\mu$. These are expressed in the following result. 

\begin{theorem}\label{thm.key}
Let $\mu$ be a probability measure on $\GL(V)$ with finite first moment and $\nu$ a non-degenerate $\mu$-stationary probability measure on $\PP(V)$. Suppose that $H_\mu$ is Zariski-connected. Then every irreducible $H_\mu$-subspace $W$ of $V$ such that $\lambda_1(W)=\lambda_1(V)$ admits a $H_\mu$-invariant complement in $V$. 
\end{theorem}
 

Before proceeding with the proof, note that when $\nu$ is a non-degenerate probability measure on $\PP(V)$ and $\pi$ is a non-zero endomorphism of $V$, since $\nu(\PP(\ker\pi))=0$, the pushforward $\pi \nu$ is a well-defined probability measure on $\PP(V)$. 
Moreover, by non-degeneracy, it is easy to see that  
\begin{equation}\label{eq.non.deg}\langle \supp(\pi \nu)\rangle=\PP(\im(\pi)),
\end{equation}
where we denoted by $\langle \supp(\pi \nu) \rangle$ the projective subspace generated by the support of $\pi \nu$. 

In the sequel, we will repeatedly use the following  lemma. It follows at once from dominated convergence theorem.
 


\begin{lemma}\label{lemma.dominated1}
Let $V$ be a vector space, $\nu$ a non-degenerate probability measure on $\PP(V)$, and $h_n$ a sequence in $\Endo(V)$ that converges to some non-zero $h\in \Endo(V)$. Then $  h_n \nu \underset{n\to +\infty}{\longrightarrow}  h \nu$ weakly.  \qed
\end{lemma}

\subsection{Consequences on the random walk}
The main outputs of this subsection are Lemma \ref{lemma.Vr} and Lemma \ref{lemma.nondegenerate}. 

We start by recalling a fundamental result of Furstenberg \cite{furstenberg.noncommuting} and Guivarc'h--Raugi \cite{guivarc.raugi.contraction} (see also \cite[II, Lemma 2.1]{bougerol.lacroix}):  for $\beta \otimes (\sum_{i=0}^{\infty}{2^{-i-1}\mu^i})$-almost every $(b,g)\in B\times \Gamma_{\mu}$, 
\begin{equation}\label{eq.martingale}b_1\cdots b_n g \nu \overset{\textrm{weakly}}{\to} \nu_b\end{equation} where $\nu_b$ is a probability measure on $\PP(V)$ such that   $\int{\nu_b d\beta(b)}=\nu$. 

An immediate consequence is the following.

\begin{lemma}\label{lemma.Vr}
Let $\mu$ be a probability measure on $\GL(V)$, $W$ a $\Gamma_\mu$-invariant subspace of $V$, and $\nu$ a non-degenerate $\mu$-stationary probability measure on $\PP(V)$. Then for $\beta$-almost every $b\in B$, every limit point $\pi_b$ of $\frac{b_1\cdots b_n}{\|b_1\cdots b_n\|}$ satisfies $\im(\pi_b)\not\subset W$. Moreover, if $\mu$ has finite first moment, then $\lambda_1(V/W) \geq \lambda_1(W)$ (or equivalently $\lambda_1(V/W)=\lambda_1(V)$). 
\end{lemma}

\begin{proof}
Since $\nu=\int{\nu_b d\beta(b)}$ and $\nu([W])=0$ (as $\nu$ is non-degenerate in $\PP(V)$), there exists $B'\subset B$ such that $\beta(B')=1$ and $\nu_b([W])=0$ for every $b\in B'$. 
By \eqref{eq.martingale}, restricting if necessary to a subset of $B'$ of $\beta$-measure $1$, we can assume $b_1\cdots b_n \nu \to \nu_b$ weakly for every $b \in B'$.  Let $b\in B'$ and $\pi_b$ a limit point of $\frac{b_1\cdots b_n}{\|b_1\cdots b_n\|}$. 
Since $\nu$ is non-degenerate in $\PP(V)$ (and $\pi_b\neq 0$), it follows from Lemma \ref{lemma.dominated1} that $b_1\cdots b_n \nu \to\pi_b \nu$. By \eqref{eq.martingale}, we deduce that $\pi_b \nu = \nu_b$. Hence the support of $\nu_b$ is included in (is actually equal to) $\im(\pi_b)$. But $\nu_b([W])=0$. Hence $\im(\pi_b)\not\subset W$. This proves the first claim. The claim about Lyapunov exponents follows from the fact that if $\lambda_1(V/W)<\lambda_1(W)$ then all limits points of $\frac{b_1\cdots b_n}{\|b_1\cdots b_n\|}$ have their image included in $W$ (a convenient way to see this is by writing a matrix representation of elements of $\Gamma_{\mu}$ in upper triangular form with the top-left block representing the action on $W$). 
\end{proof}

To state and prove the next result (Lemma \ref{lemma.nondegenerate}), we introduce some further definitions.

\begin{definition}
Let $S\subset \Endo(V)$. We denote by $C(S)$ the closed subset $\overline{\mathbb{R}_{>0}S}$ of $\Endo(V)$. 
\end{definition}

In other words, $C(S)$ is the set of all possible limits in $\Endo(V)$ of $\epsilon_n g_n$ with $\epsilon_n>0$ and $g_n\in S$. Observe also that if $\Gamma$ is a semigroup, $C(\Gamma)$ is a (closed) semigroup.
 
\begin{lemma}\label{lemma.nondegenerate}
Let $\mu$ be a probability measure on $\GL(V)$ and 
$\nu$ a non-degenerate $\mu$-stationary probability on $\PP(V)$. Then for   $\beta$-almost every $b=(b_i)_{i\in \N}\in B$, denoting $\Pi_b:=\{\pi\in \Endo(V)\setminus \{0\}: \pi \nu=\nu_b\}$, the following hold: 
\begin{enumerate}
\item 
$\Pi_b=C(\Pi_b) \setminus \{0\}$  and 
all the endomorphisms in $\Pi_b$ have the same image. \item For every $\gamma\in C(\Gamma)$ and every limit point $\pi_b$ of $\frac{b_1\cdots b_n}{\|b_1\cdots b_n\|}$  such that $\pi_b \gamma\neq 0$, we have $\pi_b\gamma\in \Pi_b$.  
\item  For every $\pi\in  \Pi_b$ and $\gamma \in C(\Gamma)$ such that $(\pi \gamma)^2\neq 0$, $\pi \gamma$ is diagonalizable over $\C$ with all of its non-zero eigenvalues having the same modulus. 
\end{enumerate}
\end{lemma}

\begin{proof}[Proof of Lemma \ref{lemma.nondegenerate}]
\begin{enumerate}

\item The first claim follows from Lemma \ref{lemma.dominated1}. The second one follows from \eqref{eq.non.deg}: the image of any element in $\Pi_b$ is equal to the subspace generated by the support of $\nu_b$.  
\item 
First we show that for $\beta$-almost every $b\in B$, the following holds for every $g\in \Gamma_{\mu}$, 
\begin{equation}\label{eq.nub.groupe}\pi_b g \nu=\nu_b.\end{equation}
Indeed, for $\beta$-almost every $b\in B$ and for all $g\in \Gamma_{\mu}$, since $\nu$ is non-degenerate and 
$\pi_b g\neq 0$, it follows from Lemma \ref{lemma.dominated1} that 
\begin{equation}\label{eq.pib}b_1\cdots b_n g \nu\to \pi_b g \nu.
\end{equation}
Combining \eqref{eq.martingale} and \eqref{eq.pib} we get that    
for $\beta\otimes (\sum_{i=1}^{\infty}{2^{-i}\mu^i})$-almost every $(b,g)\in B\times \Gamma_{\mu}$, 
$$\pi_b g \nu=\nu_b.$$
By Fubini, there exists a $\beta$-full measure subset $B'\subseteq B$ such that for every $b\in B'$ the previous identity holds for $(\sum_{i=1}^{\infty}{2^{-i}\mu^i})$-almost every  $g\in \Gamma_{\mu}$. Since, for each $b \in B'$, the set of $g$ in $\Gamma_{\mu}$ such that the previous identity holds is closed (see Lemma \ref{lemma.dominated1}) and the support of $(\sum_{i=1}^{\infty}{2^{-i}\mu^i})$ is $\Gamma_\mu$, the equality \eqref{eq.nub.groupe} follows. Using this, we deduce from (1) that $\pi_b \gamma\in \Pi_b$ whenever $\gamma\in C(\Gamma)$ and $\pi_b \gamma\neq 0$. 
 
\item  We restrict to a  subset of $B$ of $\beta$-full measure where (2) holds.  
Let now $\gamma\in C(\Gamma)$ such that $(\pi_b \gamma)^2\neq 0$. Since $\gamma, \pi_b\in C(\Gamma)$ and $C(\Gamma)$ is a semigroup, $\gamma\pi_b\gamma$ also belongs to $C(\Gamma)$ so that, by (2), $(\pi_b \gamma)^2\in \Pi_b$. 
In particular by (1), 
 $$\im(\pi_b \gamma)=\im((\pi_b \gamma)^2).$$
This implies that $\pi_b \gamma$ is not nilpotent. 
Therefore, for every $k\in \N$, $(\pi_b \gamma)^k\neq 0$ and since  $\gamma(\pi_b \gamma)^{k-1}\in C(\Gamma)$, we deduce from (2) that $(\pi_b \gamma)^k\in \Pi_b$ for every $k\in \N$. Let $(\pi_b \gamma)^{\infty}$ be any limit point of $\frac{(\pi_b \gamma)^k}{\|(\pi_b \gamma)^k\|}$. We have $(\pi_b \gamma)^{\infty}\in C(\Pi_b) \setminus \{0\}$ so that by (1),  $(\pi_b \gamma)^{\infty}\in \Pi_b$, and 
\begin{equation}\label{eq.rk.pr.inf}
\im((\pi_b \gamma)^{\infty})=\im(\pi_b \gamma).
\end{equation}
Writing $\pi_b \gamma$ in its Jordan canonical form, we deduce from \eqref{eq.rk.pr.inf} that all the non-zero eigenvalues of $\pi_b \gamma$ have the same modulus and that their respective geometric multiplicity coincide with their algebraic one, ending the proof. 
\end{enumerate}
\end{proof}

\subsection{Recurrence properties of random walks}
In this section we record a recurrence property of random matrix products which will be of crucial use for the construction of a projection map in the next section.

\begin{lemma}(Recurrence of random walks on spaces with the same Lyapunov)\label{lemma.recurrence}
Let $\mu$ be a probability measure on a topological group $G$ and 
  $\rho: G\to \GL(V)$ and $\rho': G\to \GL(V')$ two strongly irreducible representations 
  of $G$ such that $\rho_\ast\mu$ and $\rho'_\ast \mu$ have finite first moment and same top Lyapunov exponents.  
 Then for $\beta$-almost every $b\in B$, there exists a subsequence $(n_k)_{k\in \N}$ 
such that \begin{equation}\label{eq.inf} 
\inf_{k\geq 0}\frac{\|\rho(b_1\cdots b_{n_k})\|}  {\|\rho'(b_1\cdots b_{n_k})\|}>0.\end{equation}
\end{lemma}

\begin{proof}
For every $g\in G$ we denote by $\rho^t(g)\in \Endo(V^*)$ the transpose linear map on the dual space $V^*$ i.e. $\rho^t(g) f(v)=f(\rho(g)v)$. Similarly one defines  $\rho'^t(g)\in \Endo(V'^*)$. 
We equip $V^*$ and $V'^*$ with the dual norms so that $\|\rho(g)\|=\|\rho^t(g)\|$ and $\|\rho'(g)\|=\|\rho'^t(g)\|$ for every $g\in G$. 
Let $H_{\mu}$ be the subgroup of $\GL(V^*)\times \GL(V'^*)$ image of $\zeta: g\mapsto (\rho^t(g), \rho'^t(g))$. The group $H_{\mu}$ acts on $X:=\PP(V^*)\times \PP(V'^*)$.  
By compactness of $X$, we can find a   $\zeta_\ast \mu$-stationary ergodic probability measure $\nu$  on $X$.
Consider now  the dynamical system $(Y,\hat{T}, \eta) $ where  $Y:=B\times X$, $\hat{T}(b, x):=(T b,\zeta(b_1) x), \eta=\beta \otimes \nu$ and $T: B \to B$ is the shift map.  The measure $\eta$ is $\hat{T}$-ergodic \cite[Proposition 2.14]{BQ.book}. 
For every $g\in G$ and $x=([v],[v'])\in X$ we let $\sigma(g,([v],[v'])):=\log \frac{\| \rho^t(g)   v\| \|v'\|}{\|v\| \|\rho'^t(g)   v'\|}$ and  for every $b=(b_i)_{i\in \N}\in B$, 
$f(b,x):=\sigma(b_1,x)$. Denoting $q$ and $q'$ the projections from $X$ to $\PP(V^*)$ and $\PP(V'^*)$ respectively, we have 
$$\int_{X} { f d\eta}=\int_{G\times \PP(V^*)}{\log \frac{\|\rho^t(g) v\|}{\|v\|} d\mu(g) d(q_\ast \nu)([v]) - \int_{G \times \PP(V'^*)} {\log \frac{\|\rho^t(g) v\|}{\|v\|} d\mu(g) d(q'_\ast \nu)([v])}}.$$
The projection maps $q: X\to \PP(V^*)$ and $q': X\to \PP(V'^*)$ being $H_{\mu}$-equivariant, the measures $q_\ast \eta$ and $q'_\ast \eta$ are respectively $\rho^t_\ast \mu$ and $\rho'^t_\ast \mu$-stationary probability measures on $\PP(V)$ and $\PP(V^*)$. Since $H_{\mu}$ acts irreducibly on $V^*$ and $V'^*$ (because $\rho$ and $\rho'$ are irreducible representations), it follows from \cite{furstenberg.kifer} that 
all stationary probability measures in $\PP(V^*)$ (resp. $\PP(V'^*)$) have the same cocycle average. Moreover, for iid random matrix products in $\GL_d(\C)$ the Lyapunov exponents of a probability measure $\mu$ are the same as the Lyapunov exponents of its pushforward by the transpose map (see \cite{key.reversible} for a more general statement). Therefore, since by assumption the Lyapunov exponents of $V$ and $V'$ are the same, it follows that $\lambda_1(\rho^t_\ast \mu)=\lambda_1(\rho'^t_\ast \mu)$. Hence 
$$\int{f d\eta}=0.$$
By Atkinson's result \cite{atkinson} we deduce that for $\eta$-almost every $(b,x)$, 
there exists an increasing subsequence $(n_k)_{k\in \N}$ such that for every $k \in \N$ \begin{equation}\label{equation.minore}\sum_{i=1}^{n_k}{f(\hat{T}^i(b,x))}\geq -1.\end{equation}
By Fubini, we can then find $x=([v],[v'])$ such that  \eqref{equation.minore} holds for   $\beta$-almost every $b\in B$ (with $n_k$ depending on $b$). 
But $\sigma$ is a cocycle for the right action (i.e.~ $\sigma(gh,x)=\sigma(h, \zeta(g) x)+\sigma(g,x)$). Hence for every $n \in \N$ 
$$\sum_{i=1}^n {f(\hat{T}^i(b,x))}= \log \frac{\|\rho^t(b_1\cdots b_n) v\| \|v'\|}  {\|\rho'^t(b_1\cdots b_n)v'\| \|v\|}.$$ We deduce that there exists $B_1\subset B$ such that $\beta(B_1)=1$ and for every $b\in B_1$, 
 there exists a subsequence $(n_k)_{k\in \N}$ such that 

\begin{equation}\label{eq.inf1} \inf_{k\geq 0}\frac{\|\rho^t(b_1\cdots b_{n_k}) v\|  }  {\|\rho'^t(b_1\cdots b_{n_k})v'\|  }>0.\end{equation}
On the other hand,  the representations $\rho$ and $\rho'$ are strongly irreducible, so are their transpose maps. Hence 
  by \cite[III, Proposition 3.2 (c)]{bougerol.picard},  there exists $B_2\subset B$ such that $\beta(B_2)=1$ and for every $b\in B_2$, 

\begin{equation}\label{eq.inf2}\inf_{n\in \N}\frac{\|\rho^t(b_1\cdots b_{n }) v\|}{\|\rho^t(b_1\cdots b_n)\|} > 0 \qquad \text{and} \qquad \inf_{n\in \N}\frac{\|\rho'^t(b_1\cdots b_{n }) v\|}{\|\rho'^t(b_1\cdots b_n)\|}>0 .\end{equation}
Combining \eqref{eq.inf1} and \eqref{eq.inf2}, we deduce that for every $b\in B_1\cap B_2$, there exists a subsequence $(n_k)_{k\in \N}$ such that \eqref{eq.inf} holds with $\rho$ and $\rho'$ replaced by $\rho^t$ and $\rho'^t$ respectively. Hence \eqref{eq.inf} holds too.  
\end{proof}

\subsection{Consequence on the semigroup generated by the support of $\mu$}
The goal of this part is to prove a key technical ingredient (Proposition \ref{prop.ellis}) for the proof of Theorem \ref{thm.key}. 
We introduce the following notation. 
To any pair $(S, \pi)$ where $S$ is a subset of  $\Endo(V)$ and $\pi$ an endomorphism of $V$, we  associate the subset $$\pi S \pi=\{\pi \gamma \pi | \gamma\in S\}$$ of $\Endo(V)$ whose restriction to $\im(\pi)$ yields  a subset of $\Endo(\im(\pi))$. We  denote by $\langle \pi S \pi \rangle$ the semigroup of $\GL(V)$ generated by $\pi S \pi$. 

The key technical result for the proof of Theorem \ref{thm.key} is the following. 
\begin{proposition}\label{prop.ellis}
Let $\mu$ be a  probability measure on $\GL(V)$ with finite first moment and $\nu$ a non-degenerate $\mu$-stationary probability on $\PP(V)$.  Let $W_0=\{0\}\subset W_1\subset  \cdots \subset W_{r-1}\subset W_r=V$ be a  Jordan–H\"{o}lder decomposition of $V$  as $\Gamma_\mu$-module. For every $i \in \{1,\cdots, r\}$ such that $\lambda_1(W_i/W_{i-1})=\lambda_1(\mu)$, 
there exists $\pi \in C(\Gamma_\mu)$ such that $\pi^2=\pi$,  $\pi(W_i)\not\subset W_{i-1}$, $\im(\pi)\not\subset W_{r-1}$ and  there exists a non-empty Zariski-open subset $O$   of $H_\mu$ such that $\langle \pi(\Gamma_\mu \cap O)\pi\rangle_{|_{\im(\pi)}}$ is a subsemigroup of $\GL(\im(\pi))$ whose projection to $\PGL(\im(\pi))$ has compact closure. 
\end{proposition}

The endomorphism $\pi$ will be constructed using random walks.
  
\begin{proof}
The semigroup $C(\Gamma_\mu)>\Gamma_\mu$ stabilizes the Jordan–H\"{o}lder decomposition of $V$. Hence it acts naturally on each quotient $W_k/W_{k-1}$ inducing a homomorphism $\rho_k: C(\Gamma_\mu) \to \Endo(W_k/W_{k-1})$.  
Let $i\in \{1,\cdots, r\}$ such that $\lambda_1(W_i/W_{i-1})=\lambda_1(V)$. By Lemma \ref{lemma.Vr}, we have $\lambda_1(V/W_{r-1})=\lambda_1(V)=\lambda_1(W_i/W_{i-1})$. We can then find a $\beta$-generic $b\in B$ satisfying simultaneously all the conclusions of Lemmas \ref{lemma.Vr}, \ref{lemma.nondegenerate}  and \ref{lemma.recurrence} applied to the representations of $H_{\mu}$  on $W_i/W_{i-1}$ and $V/W_r$ (these representations are strongly irreducible because they are irreducible and $H_\mu$ is Zariski-connected).  
\begin{enumerate}[label=(\roman*)]
\item \label{roman1}
By Lemma \ref{lemma.recurrence} there exists a subsequence $(n_k)_{k\in \N}$ such that $\inf_{k\in \N}{\frac{\|\rho_i(b_1\cdots b_{n_k})\|}{\|\rho_r(b_1\cdots b_{n_k})\|}}>0$. Passing to a further subsequence we can assume without loss of generality that $\frac{b_1\cdots b_{n_k}}{\|b_1\cdots b_{n_k}\|}$ converges to an endomorphism $\pi_b$ of $V$ as $k \to \infty$. Clearly, $\pi_b \in C(\Gamma_\mu)$ and by Lemma \ref{lemma.Vr}, $\rho_r(\pi_b)\neq 0$. Hence 
$\inf_{k\in \N}{\frac{\|\rho_r(b_1\cdots b_{n_k})\|}{\|b_1\cdots b_{n_k}\|}}>0$ and consequently, $\inf_{k\in \N}{\frac{\|\rho_i(b_1\cdots b_{n_k})\|}{\|b_1\cdots b_{n_k}\|}}>0$. This implies that $\rho_i(\pi_b)\neq 0$. 
\item \label{roman2} 
For $k\in \{i,r\}$, it follows from \ref{roman1} and the irreducibility of the action of $\Gamma_\mu$ on $W_k/W_{k-1}$ that the set $S_k:=\{g\in H_\mu : g \im(\rho_k(\pi_b)))\subset \ker(\rho_k(\pi_b))\}$
is a proper Zariski-closed subvariety of $H_\mu$. 
By Zariski-connectedness of the group $H_\mu$ (or equivalently its topological irreducibility), $S_i\cup S_r$ is a proper closed subvariety of $H_\mu$. Since $\Gamma_\mu$ is Zariski-dense in $H_\mu$, we deduce the existence of an element $g\in \Gamma_\mu$ satisfying $\rho_k((\pi_b g)^2)\neq 0$ for $k\in \{i,r\}$. In particular, we have $(\pi_b g)^2\neq 0$.



\item \label{roman3} In view of \ref{roman2}, we are in a position to apply Lemma \ref{lemma.nondegenerate} to $\pi_b g$, and (using its notation) deduce that $\pi_b g\in \Pi_b$ and that $\pi_b g$ is a non-zero endomorphism of $V$ diagonalizable over $\C$ with all of its non-zero eigenvalues having the same modulus, say $r>0$. Since $C(\Gamma_\mu)$ is a semigroup, $\pi_b g\in C(\Gamma_\mu)$. Hence, $\pi_bg$ preserves each $W_k$ and the endomorphism that it induces on each quotient $W_k/W_{k-1}$ has either the same property (all non-zero eigenvalues have modulus equal to $r$) or is equal to zero. In particular, for every $n \geq 0$, 
$(\pi_bg)^n \neq 0$ and by Lemma \ref{lemma.nondegenerate} (2), $(\pi_bg)^n \in \Pi_b$. 

\item \label{roman4} Let $r e^{i\theta_1}, \cdots r e^{i\theta_s}$  be the  non-zero eigenvalues of $\pi_b g$. 
We can find a sequence $(m_k)$ such that $\theta_j^{m_k}\underset{k\to +\infty}{\longrightarrow}0 \; (\text{mod}\; 2\pi)$ for every $j\in \{1,\cdots, s\}$. 
Hence, we can find a limit point $\pi_b^{\infty}$ of $r^{-n}(\pi_b g)^n$ which is a projection endomorphism verifying for every $k\in \{1,\cdots, r\}$, $\rho_k(\pi_b^\infty)\neq 0$
if and only if  $\rho_k(\pi_b g)\neq 0$. Since, by \ref{roman2}, the latter condition is verified for $k\in \{i,r\}$, this gives 
$\pi_b^{\infty}(W_k)\not\subset W_{k-1}$.
\item \label{roman5}
From now on, we set $\pi:=\pi_b^{\infty}$. 
Observe that, since by Lemma \ref{lemma.nondegenerate} (1) $\Pi_b$ is closed in $\Endo(V) \setminus \{0\}$ and $\pi$ is a non-zero limit point of $r^{-n}(\pi_b g)^n$'s which are all in $\Pi_b$ by \ref{roman3}, it follows that $\pi\in \Pi_b\setminus \{0\}$. Moreover, $\pi\in C(\Gamma_{\mu})$. 

\item \label{roman6} Let $O$ be the set of elements $h\in H_\mu$ such that $\pi h \pi\neq 0$. It is a  Zariski-open subset of $H_\mu$ and it is non-empty as it  contains the identity element (as $\pi$ is a non-zero projection by \ref{roman4}).

\item \label{roman7} We now claim that $\pi$ plays the same role as $\pi_b$   in Lemma \ref{lemma.nondegenerate} (2), in other words, for any $\gamma\in C(\Gamma_\mu)$ such that $\pi \gamma \neq 0$, we have $\pi\gamma\in \Pi_b$. Indeed, let $\gamma\in C(\Gamma_\mu)$ be any endomorphism such that $\pi \gamma\neq 0$. We have $\pi=\lim_{k\to \infty}{r^{-m_k}(\pi_b g)^{m_k}}$ for some sequence $m_k$ and hence for sufficiently large $k$, $(\pi_b g)^{m_k}\gamma \neq 0$ so that by Lemma \ref{lemma.nondegenerate} (2) $r^{-m_k}(\pi_b g)^{m_k} \gamma=r^{-m_k}\pi_b (g\pi_b)^{m_k-1} g \gamma\in \Pi_b$. Hence $\pi \gamma\in C(\Pi_b)\setminus \{0\}$. By (1) of the same lemma, $\pi\gamma\in \Pi_b$ as claimed. 

\item \label{roman8}  Let $\gamma\in O \cap \Gamma_\mu$. Since $C(\Gamma_\mu)$ is a semigroup, and by \ref{roman5} $\pi \in C(\Gamma_\mu)$, we have $\gamma \pi\in C(\Gamma_\mu)$. It follows from \ref{roman7} and the definition of $O$ that $\pi \gamma \pi\in \Pi_b\setminus \{0\}$. By Lemma \ref{lemma.nondegenerate} (1)  $\im(\pi\gamma \pi)=\im(\pi)$.  But since $\pi$ is a projection,  $\im((\pi \gamma \pi)_{|{\im(\pi)}})=\im(\pi \gamma \pi)$.  Hence the restriction of $\pi\gamma \pi$ to $\im(\pi)$ is an invertible endomorphism of $\im(\pi)$.
In other words $\langle \pi (\Gamma_\mu \cap O)\pi \rangle_{|_{\im(\pi)}}\subset \GL(\im(\pi))$.
\item \label{roman9} We claim that $\nu_b$ is invariant under the action of $\langle \pi (\Gamma_\mu \cap O) \pi \rangle$. 
Indeed, it follows from \ref{roman8} that $0\not\in \langle \pi  (\Gamma_\mu \cap O) \pi \rangle$ and that   $\langle \pi (\Gamma_\mu \cap O) \pi \rangle \subset C(\Gamma_\mu)$. Since each element of $\langle \pi (\Gamma_\mu\cap O) \pi\rangle$  is invariant by left multiplication by $\pi$,   \ref{roman7} yields that   $$\langle \pi (\Gamma_\mu \cap O) \pi \rangle\subset \Pi_b\setminus \{0\}.$$
Let $\gamma\in \langle \pi (\Gamma_\mu \cap O) \pi \rangle  $. We have $\gamma \nu=\nu_b$. 
But by \ref{roman5} $\pi \nu=\nu_b$. Since $\pi$ is a projection and $\gamma$ is also invariant by right multiplication by $\pi$, we deduce that $\gamma\nu_b=\nu_b$.

\item Now we conclude the proof. It only remains to show that the closure of $\langle \pi \Gamma_\mu \pi\rangle_{|_{\im(\pi)}}$  in $\textrm{PGL}(\im(\pi))$ is a compact group (indeed the projection map $\pi$ was constructed in \ref{roman4} and it satisfies the claims made about its image in the statement, and the Zariski open subset $O$ of $H_{\mu}$ was constructed in \ref{roman6}). To see the compactness, first observe that  $\nu_b$ is a non-degenerate probability measure on $\PP(\im(\pi))$, as it follows from the fact that $\pi \in \Pi_b\setminus \{0\}$ and the non-degeneracy of $\nu$ (see \eqref{eq.non.deg}). 
By a classical result of Furstenberg \cite{furstenberg.noncommuting} (see also \cite[Corollary 3.2.2]{zimmer.book}), the stabilizer in $\PGL(E)$ (for any vector space $E$) of any non-degenerate probability measure on $\PP(E)$ is a compact group. Hence $\textrm{Stab}(\nu_b)$ is a compact group. It follows from \ref{roman8} and \ref{roman9},   that  $\langle \pi (\Gamma_\mu\cap O) \pi\rangle_{|_{\im(\pi)}}$ is a semigroup included in a compact group. Its closure is then a compact group and the proof is complete.

 \end{enumerate}

\end{proof}
\subsection{Proof of Theorem \ref{thm.key}}\label{subsec.thm.key}

The proof is in two steps.\\

(i) We first show the following particular case: every proper irreducible subspace admits a non-zero invariant subspace that is in direct sum. To prove this,
let $\nu$ be a non-degenerate $\mu$-stationary probability measure on $\PP(V)$ and $W$ a proper $\Gamma_{\mu}$-irreducible invariant subspace of $V$ with top Lyapunov exponent, as in the statement. 
We can find a Jordan–H\"{o}lder decomposition $\{0\}=W_0\subset W_1\subset \cdots \subset W_r=V$ of $V$ such that $W_1=W$ (necessarily $r\neq 1$ as $W$ is proper). 
Proposition \ref{prop.ellis} applies with $i=1$, and it yields a projection $\pi \in C(\Gamma_\mu)$ such that $\pi(W)\neq \{0\}$ and  $\im(\pi)\not\subset W$,   and   a Zariski-open subset $O$ of $H_\mu$ such that $\overline{\pi(O\cap \Gamma_\mu) \pi}$ is a compact subgroup $H$ of $\GL(\im(\pi))$. 
In particular, $\im(\pi)$ is $H$-completely reducible. But  $\pi(W)=W\cap \im(\pi)$ is $H$-invariant (because, being in $C(\Gamma_\mu)$, $\pi$ stabilises $W$). Hence there exists a direct $H$-invariant complement $W'$ of $\pi(W)$  in $\im(\pi)$. The $H$-invariance of $W'$ is equivalent to saying that 
 
\begin{equation}\label{eq.petits.zeros}
\forall \gamma\in O\cap \Gamma_\mu, \gamma W'\subset \pi^{-1}(W'). \end{equation}

Since $H_\mu$ is Zariski-connected, $\Gamma_\mu\cap O$ is Zariski-dense in $H_\mu$ and since the set of elements $\gamma\in H_\mu$ verifying $\gamma W'\subset \pi^{-1}(W')$
is Zariski-closed, it follows from \eqref{eq.petits.zeros} that 
\begin{equation}\label{eq.grands.zeros}
\forall g\in H_\mu, \quad  g W'\subset  \pi^{-1}(W'). \end{equation}
Consider now the $H_\mu$-invariant subspace $E:=\textrm{Span}(\{g v :  g\in H_\mu, v\in W'\})$. By \eqref{eq.grands.zeros}, we have $W\not\subset E$. 
Since $W$ is irreducible, we deduce that $W\cap E=\{0\}$. 
Finally $E\neq 0$ as otherwise $W'=\{0\}$, contradicting    $\im(\pi)\neq W$.  This  finishes
 the claim at the beginning of the proof.\\ 

(ii)
Now we prove the theorem in full generality. We proceed by induction on $\textrm{dim}(V)$. The result trivially holds when $\textrm{dim}(V)=1$. Let $V,\mu, \nu$, and $W$ be as in Theorem \ref{thm.key}. By step (i), there exists a non-zero $H_\mu$-invariant subspace $E$ of $V$  which is in direct sum with $W$.  Let $V':=V/E$.  Since $\nu([E])=0$ (as $\nu$ is non-degenerate and $E$ is proper), $\nu$ descends to a $\mu$-stationary probability measure $\overline{\nu}$ on $\PP(V/E)$ which is non-degenerate.  Let $q: V\to V/E$ denote the projection map.  The subspace $q(W)$ of $V/E$ is a $H_\mu$-invariant subspace of $V/E$ which is $H_\mu$-equivariantly isomorphic to $W$ (as $W\cap E=\{0\}$). Thus $q(W)$ is irreducible and $\lambda_1(q(W))=\lambda_1(\mu)=\lambda_1(V/E)$. The induction hypothesis applied with the ambient vector space $V/E$ (satisfying $\dim(V/E)<\dim(V)$), invariant subspace $q(W)$, non-degenerate stationary measure $\overline{\nu}$ and random walk measure $\mu$ yields a $G_{\mu}$-invariant  complement $\tilde{F}$ of $q(W)$ in $V/E$. 
Let $F:=q^{-1}(\tilde{F})$. This is a $G_{\mu}$-invariant subspace of $V$. Clearly $V=W+F$ and, since $W\cap E=\{0\}$, we have $W\cap F=\{0\}$, concluding the proof. \qed


\section{Proof of the main statements}\label{sec.last}

\subsection{Proof of Theorem \ref{thm.main.Lmu=0}}
By \S \ref{section.preliminary} it is enough to prove Theorem \ref{thm.inter}. 
Let $\mu$ be a probability measure on $\GL(V)$ such that $F_2(V)=\{0\}$ and assume that $\Gamma_{\mu}$ is Zariski-connected. Let $\nu$ be a non-degenerate $\mu$-stationary ergodic probability measure on $\PP(V)$. Consider an   irreducible $\Gamma_\mu$-invariant subspace $W$ of $V$. By the assumption $F_2(V)=\{0\}$, we have $\lambda_1(W)=\lambda_1(V)$ and hence Theorem \ref{thm.key} provides an invariant complement $W'$ of $W$ in $V$. We have therefore shown that any irreducible subspace has an invariant complement and this implies complete reducibility. \qed 

\begin{remark}[$V_\nu$ is a sum of same highest weight representations]\label{rk.same.weight}
Since by Corollary \ref{corol.critical.homogene} the support of $\nu$ lives in a compact $H_\mu$-orbit, every irreducible $H_\mu$-subrepresentation of $V_\nu$ has the same highest weight.

 \end{remark}

\subsection{Proof of Theorem \ref{thm.with.W}}
As mentioned in the introduction, Theorem \ref{thm.with.W} is a slightly stronger version of Theorem \ref{thm.main.Lmu=0} (which additionally combines results from our previous work \cite{aoun.sert.stationary1}).

To prove it, we will begin by treating the particular case that does not involve noise from several  Furstenberg--Kifer--Hennion exponents (FKH exponents for short) of $W$. 
Let us make these terms precise. For a  probability measure $\mu$ on $\GL(V)$  with finite first moment and a probability measure $\nu$ on a $\PP(V)$ the cocycle average of $\nu$ is the  scalar $\alpha(\nu):=\iint{\log \frac{\|g v\|}{\|v\|} d\mu(g) d\nu(\R v)}$. 
A result of Furstenberg--Kifer \cite{furstenberg.kifer} says that, when $\nu$ runs over all $\mu$-stationary ergodic probability measures on $\PP(V)$, $\alpha(\nu)$ can take finitely many values $\lambda_1(V)=\beta_1>\beta_2>\cdots>\beta_k=\beta_{\min}(V)$ and these all are top Lyapunov exponents of the FKH spaces $F_i<V$ appearing in Theorem \ref{thm.main} (and top Lyapunov exponents of the space $V_\nu$ generated by the support of $\nu$ which is then included in the corresponding $F_i$). These exponents $\beta_i$'s are called FKH exponents\footnote{Sometimes also called deterministic exponents.}. Saying that $F_2(V)=\{0\}$ (i.e.~$V$ is critical or equivalently $\beta_{\min}(V)=\lambda_1(V)$) is equivalent to saying that there is only one cocycle average. 

The particular case of Theorem \ref{thm.with.W} that we will start with (Theorem \ref{thm.pure.expansion} below), corresponds to $\lambda_1(V_{\overline{\nu}})=\beta_{\min}(W)$. This is the case for instance when $\Gamma_\mu$ acts irreducibly on both of $W$ and $V/W$ in Theorem \ref{thm.with.W}.

\begin{theorem}[Purely critical case]\label{thm.pure.expansion}
Let $\mu$ be a probability measure on $\GL(V)$ with  finite first moment and $W$ a $\Gamma_\mu$-invariant subspace. Let $\overline{\nu}$ be a $\mu$-stationary and ergodic probability measure on $\PP(V/W)$ such that 
\begin{equation}\label{eq.antidom.sec}
\alpha(\overline{\nu}) \leq \beta_{\min}(W).   
\end{equation}
Then, the following are equivalent: 
\begin{itemize}
    \item[(i)] There exists a $\mu$-stationary lift $\nu$ of $\overline{\nu}$ on $\PP(V)\setminus \PP(W)$. 
\item[(ii)] There exists a $\Gamma_{\mu}$-invariant subspace $W'$ of $V$ in direct sum with $W$ such that $\PP(W'/W)$ is the projective subspace generated by $\overline{\nu}$.
\end{itemize}
\end{theorem}

First, we begin by recording the following consequence of \cite[Theorem 1.5]{aoun.sert.stationary1}.

\begin{lemma}[Lift invariance of cocycle-average]\label{lemma.cocycle.invariance}
Let $\mu$ be a probability measure with finite first moment on $\GL(V)$, $W$ a $\Gamma_{\mu}$-invariant subspace of $V$ and $\overline{\nu}$ a $\mu$-stationary ergodic probability measure on $\PP(V/W)$. If $\nu$ is a $\mu$-stationary 
lift of $\overline{\nu}$ then $\alpha(\overline{\nu})=\alpha(\nu)$. 
\end{lemma}

\begin{proof}
It is clear that $\alpha(\nu)\geq \alpha(\overline{\nu})$. Without loss of generality, we can suppose $\nu$ to be ergodic. 
The probability measure $\nu$ lives in the projective space of $\pi^{-1}(F_{\overline{\nu}})\supset W$ where $\pi: V\to V/W$ is the canonical projection. Hence without loss of generality we can assume that $V=\pi^{-1}(F_{\overline{\nu}})$ so that $\lambda_1(V/W)=\alpha(\overline{\nu})$. 
 Arguing by contradiction, assume that $\alpha(\nu)>\lambda_1(V/W)$. 
 By ergodicity of $\nu$, it follows from \cite[Theorem 3.9]{furstenberg.kifer} that  $\alpha(\nu)$ is some Lyapunov exponent of $V$. But by \cite[Lemma 3.6]{furstenberg.kifer}  $\lambda_1(V)=\max\{\lambda_1(V/W),\lambda_1(W)\}$, necessarily $\alpha(\nu)$ is some top Lyapunov exponent of $W$ so that $\lambda_1(W)>\alpha(\overline{\nu})$. It is enough to apply now Theorem \cite[Theorem 1.5]{aoun.sert.stationary1}. 
\end{proof}

\begin{proof}[Proof of Theorem \ref{thm.pure.expansion}]
By \cite[Theorem 1.5]{aoun.sert.stationary1} it suffices to treat the case  $\alpha(\overline{\nu})=\beta_{\min}(W)$.  The implication $(i)\Longrightarrow (ii)$ is trivial. To prove the other implication, suppose that $\nu$ is some $\mu$-stationary lift of $\overline{\nu}$ on $\PP(V)\setminus \PP(W)$.
Denote for simplicity $V':=V_{\nu}$ and $W':=V_{\nu}\cap W$. 
By  Lemma \ref{lemma.cocycle.invariance}, $\alpha(\nu)=\beta_{\min}(W)$  so that $W'\subset F_k(W)$ where $F_k(W)$ is the smallest FKH space of $W$.  By definition of the smallest FKH exponent $\beta_{\min}(W)$, it follows that $F_2(W')=\{0\}$.
Since $\beta_{\min}(W)=\lambda_1(V')$, we deduce that $F_2(V')\cap W'=\{0\}$. 
Since $\nu(F_2(V'))=0$,  $V'/F_2(V')$ is the projective space generated by $\tilde{\nu}$, the projection of $\nu$ on $V'/F_2(V')$. 
Applying now   Theorem \ref{thm.main.Lmu=0} to $V'/F_2(V')$, we deduce that the subspace 
$(W'+F_2(V'))/F_2(V')$ admits a $\Gamma_{\mu}$-invariant complement in $V'/F_2(V')$. Since $F_2(V')\cap W'=\{0\}$, lifting this invariant complement to $V'$ gives rise to an invariant complement $U$ of $W'$ inside $V'$. Clearly $U\cap W=\{0\}$. Moreover 
$U\simeq V'/W'$ is isomorphic as $\Gamma_{\mu}$-space to $V_{\overline{\nu}}$. 
\end{proof}

\begin{proof}[Proof of Theorem \ref{thm.with.W}]
The deduction of the general case (Theorem \ref{thm.with.W}) from the purely critical case (Theorem \ref{thm.pure.expansion}) is done in the same way as in  the proof \cite[Theorem 1.5]{aoun.sert.stationary1} using \cite[Theorem 5.1]{aoun.sert.stationary1}.   We indicate  the arguments. The implication $(ii) \implies (i)$ follows immediately from \cite[Theorem 1.1]{aoun.sert.stationary1} (contracting case) applied to $W'$ as ambient space and $W'\cap W$ as invariant subspace. Now we prove $(i) \implies (ii)$. Let $\overline{\nu}$ be a $\mu$-stationary ergodic probability measure on $\PP(V/W)$ such that $\alpha(\overline{\nu})\leq \lambda_1(W)$. Denote by $j$ the smallest index such that 
$\beta_{j+1}(W)<\alpha(\overline{\nu})\leq \beta_j(W)$. Without loss of generality we can assume $V/W=V_{\overline{\nu}}$. 
Projecting $\nu$ to $\PP(V/F_{j+1}(W))$, we get a $\mu$-stationary measure $\tilde{\nu}$ on $\PP(V/F_{j+1}(W))\setminus \PP(W/F_{j+1}(W))$ that lifts $\overline{\nu}$. Since $\beta_{\min}(W/F_{j+1}(W))=\beta_j(W)$ (see for instance \cite[Corollary 3.4]{aoun.sert.stationary1}), our assumption on $\overline{\nu}$ yields $\beta_{\min}(W/F_{j+1}(W))\geq \alpha(\overline{\nu})$.  Theorem \ref{thm.pure.expansion} applied to $V/F_{j+1}(W)$ as ambient space and $W/F_{j+1}(W)$ as subspace yields a $\Gamma_{\mu}$-invariant complement $\tilde{W}$ of $W/F_{j+1}(W)$ in $V/F_{j+1}(W)$. 
Its preimage $W'<V$   by the projection map $V\to V/F_{j+1}(W)$ is a $\Gamma_{\mu}$-invariant subspace whose intersection with $W$ is equal to $F_{j+1}(W)$ (hence of Lyapunov exponent $\beta_j(W)<\alpha(\overline{\nu}))$ and such that $V_{\overline{\nu}}=W'/W'\cap W$.
\end{proof}

\begin{remark}[On uniqueness of lifts]\label{rk.not.unique}
Keep the assumptions of Theorem \ref{thm.with.W}. Suppose that there exists a $H_\mu$-invariant subspace $W'$ that intersects   $W$  only in an invariant subspace of slower expansion and such that $\overline{\nu}$ is supported in $\PP(W'/W'\cap W)$. 
\begin{enumerate}
\item Suppose $\lambda_1(W)>\alpha(\overline{\nu})$.  Then   $\overline{\nu}$ admits a unique $\mu$-stationary lift $\nu$ on $\PP(V)\setminus \PP(W)$ and the subspace generated by $\nu$ is $\PP(W')$ \cite[Theorem 1.5]{aoun.sert.stationary1}. 
\item  Suppose that $\lambda_1(W)=\lambda(V/W)$. 
Although $\PP(W')$ admits a unique $\mu$-stationary lift of $\overline{\nu}$, $\overline{\nu}$ may admit many lifts on $\PP(V)\setminus \PP(W)$ and these may even be non-degenerate.
An example is a Zariski-dense probability measure $\mu$ on $H=\SO_2(\R)\times \SO_2(\R)$ acting on $V=\R^2\oplus \R^2$. For any $x\in \PP(V)$ that does not belong to one of the two invariant $2$-planes, the $H$-orbit of $x$ is a proper algebraic variety in $\PP(V)$ which is not included in any proper projective subspace. Being compact, each such orbit supports a (unique) stationary probability measure.
\end{enumerate}
\end{remark}

\begin{remark}[Affine case]
We illustrate Theorem \ref{thm.with.W} by showing how it recovers Bougerol--Picard's result \cite{bougerol.picard} in the invertible case:  if $\mu$ is a probability measure on the affine group $\textrm{Aff}(\R^d)$ such that $\Gamma_\mu$ does not preserve an affine subspace of $\R^d$, then there exists a $\mu$-stationary probability measure on $\R^d$ if and only if the top Lyapunov exponent of the linear part is negative. 
The backward implication (contracting case) being standard, we comment on the forward implication. It is enough to embed $\textrm{Aff}(\R^d)$ in $\GL_{d+1}(\R)$ in the usual way and to take $W$ as the hyperplane in $V:=\R^{d+1}$ spanned by the first $d$  vectors of the canonical basis. In this case the action on the quotient $V/W$ is trivial so  $\alpha(\overline{\nu})\leq \lambda_1(W)\Longleftrightarrow 0\leq \lambda_1(W)$. Moreover stationary measures on $\R^d$ (for the affine action) correspond to stationary measures on $\PP(V)\setminus \PP(W)$ (for the linear action). Affine irreducibility is equivalent to saying that any $\Gamma_\mu$-invariant subspace of $V$ must be included in $W$. In particular, a subspace $W'$ fulfilling  condition (ii) of Theorem \ref{thm.with.W} cannot exist, showing indeed the absence of stationary probability measures on $\R^d$ in the critical/expanding situation. 

Finally, one also recovers \cite[Theorem 5.1]{bougerol.tightness} in a similar way (realizing the linear action as a projective action). The extra conclusion of compactness in \cite[Theorem 5.1]{bougerol.tightness} follows from Corollary \ref{corol.critical.homogene} (or more simply transience of random walks on non-compact semisimple Lie groups, due to Furstenberg \cite{furstenberg.noncommuting}). We omit the details. 
\end{remark}

\subsection{Proof of Corollary \ref{corol.homogeneous}}
First we show $(2.1)\Longrightarrow (2.2)$ together with the uniqueness statement (1). Denote $F_\mu:=F_2(V_{\mathcal{O}})$. Suppose that there exists a $\mu$-stationary probability measure $\nu$ on $\mathcal{O}$. We can suppose that $\nu$ is ergodic, and since $\mathcal{O} \cap \PP(F_\mu)=\emptyset$, $\nu(\PP(F_\mu))=0$. Let $\psi$ be the natural $H_\mu$-equivariant projection $\PP(V_\mathcal{O}) \setminus \PP(F_\mu) \to \PP(V_\mathcal{O}/F_\mu)$ and denote   $\overline{\nu}:=\psi_\ast \nu$. This is a $\mu$-stationary probability measure on the $H_\mu$-orbit $\psi(\mathcal{O})$ in $\PP(V_\mathcal{O}/F_\mu)$. 
Observe that by \cite{furstenberg.kifer},  $F_2(V_\mathcal{O}/F_\mu)=\{0\}$ (in other words $V_\mathcal{O}/F_\mu$ is critical). Therefore Corollary \ref{corol.critical.homogene} implies that $\psi(O)$ is closed. Now we show that $\overline{\mathcal{O}}\setminus  \mathcal{O}\subset \PP(F_\mu)$. We argue by contradiction. Suppose this is not the case. Since $\psi(\mathcal{O})$ is closed, this implies that $\psi(\mathcal{\overline{O}}\setminus (O\cup \PP(F_\mu))$ is a non-empty $H_{\mu}$-invariant subset of the orbit $\psi(O)$.
Hence $\psi(\mathcal{\overline{O}}\setminus (O\cup \PP(F_\mu))=\psi(\mathcal{O})$. 
Consider now a $\overline{\nu}$-generic point $\overline{x}\in \psi(\mathcal{O})$, i.e.~ $\frac{1}{n}\sum_{i=1}^n{\mu^i \ast \delta_{\overline{x}}}$ converges weakly to $\overline{\nu}$. We can find $x\in \overline{\mathcal{O}}\setminus (\mathcal{O}\cup \PP(F_\mu))$ such that $\psi(x)=\overline{x}$. 
By \cite[Proposition 1.2]{aoun.sert.stationary1} it follows that  $\frac{1}{n}\sum_{i=1}^n{\mu^i \ast \delta_x}\to \nu$.  Since the orbit $\mathcal{O}$ is locally closed (see for instance \cite[8.3]{humphreys}) and  $x \in \overline{\mathcal{O}}\setminus \mathcal{O}$, it follows from the equidistribution above that $\nu(\overline{\mathcal{O}}\setminus \mathcal{O})=1$, which contradicts  $\nu(\mathcal{O})=1$. This shows that $(2.1) \Longrightarrow (2.2)$. 
Now we show the uniqueness of $\nu$. 
Clearly $V_{\psi(\mathcal{O})}=V_{\mathcal{O}}/F_\mu$. Hence by Theorem \ref{thm.main.Lmu=0}, $V_{\mathcal{O}}/F_\mu$ is completely reducible. Hence the  uniqueness of $\nu$ follows directly from the uniqueness of a $\mu$-stationary probability measure on $\psi(O)$, by \cite[Theorems 1.5 \& 1.7]{BQ.projective}, and then by the uniqueness
of lift statement in \cite[Theorem 1.1]{aoun.sert.stationary1}. 

Conversely, assume that $\psi(\mathcal{O})$ is compact and that $\overline{\mathcal{O}}\setminus \mathcal{O} \subset F_\mu$. By compactness of $\psi(\mathcal{O})$ there exists a $\mu$-stationary probability measure $\overline{\nu}$ on $\psi(\mathcal{O})$.  Let $x\in \mathcal{O}$ such that $\psi(x)$ is a $\overline{\nu}$-generic point (i.e.~$\frac{1}{n} \sum_{k=1}^n \mu^{k} \ast \delta_{\psi(x)} \to \overline{\nu}$ as $n \to \infty$). By \cite[Theorem 1.1]{aoun.sert.stationary1} there exists a unique $\mu$-stationary probability measure $\nu$ on $\PP(V_\mathcal{O}) \setminus \PP(F_\mu)$ lifting $\nu$ and, by \cite[Proposition 1.2]{aoun.sert.stationary1}, $x$ is generic for $\nu$. It follows that $\nu$ is supported on $\overline{\mathcal{O}}$. 
Since $\nu(\PP(F_\mu))=0$ and $\overline{\mathcal{O}}\setminus \mathcal{O}\subset F_\mu$, we get that $\nu(\mathcal{O})=1$.  \qed

\end{document}